\documentclass[12pt,a4paper,psamsfonts]{amsart}
\usepackage{amssymb,amscd,amsxtra,calc}
\usepackage{cmmib57}

\theoremstyle{plain}
    \newtheorem{thm}{Theorem}[section]

    \newtheorem{lemma}[thm]{Lemma}
    
    \newtheorem{conjecture}[thm]{Conjecture}
    
    \newtheorem{theorem}[thm]{Theorem}

\theoremstyle{definition}
    \newtheorem{definition}[thm]{Definition}

\theoremstyle{remark}

\title[]{Convergence to minima for the continuous version of Backtracking Gradient Descent}

\author{Tuyen Trung Truong}
\address{Matematikk Institut, Universitetet i Oslo, Blindern, 0851 Oslo, Norway}
\email{tuyentt@math.uio.no}
\thanks{}
\date{\today}
\begin{document}
\begin{abstract}
Lee et al. and Panageas and Piliouras showed that if $f:\mathbb{R}^k\rightarrow \mathbb{R}$ is both a $C^2$ and $C^{1,1}_L$ function, and if $\delta <1/L$, then there is a set $\mathcal{E}$ of Lebesgue measure $0$ so that if $x_0\in \mathbb{R}^k\backslash \mathcal{E}$, then the Standard Gradient Descent (Standard GD) process $x_{n+1}=x_n-\delta \nabla f(x_n)$, {\bf if converges}, cannot converge to a {\bf generalised} saddle point. (Remark: for the convergence of $\{x_{n}\}$, more assumptions are needed, see the appendix of this paper.) In this paper, we prove a stronger result when replacing the Standard GD by a continuous version of Backtracking GD. More precisely, we have:

{\bf Theorem.} Let $f:\mathbb{R}^k\rightarrow \mathbb{R}$ be a $C^{1}$ function, so that $\nabla f$ is locally Lipschitz continuous. Assume moreover that $f$ is $C^2$ near its generalised saddle points. Fix real numbers $\delta _0>0$ and $0<\alpha <1$. Then there is a smooth function $h:\mathbb{R}^k\rightarrow (0,\delta _0]$ so that the map $H:\mathbb{R}^k\rightarrow \mathbb{R}^k$ defined by $H(x)=x-h(x)\nabla f(x)$ has the following property: 

(i) For all $x\in \mathbb{R}^k$, we have $f(H(x)))-f(x)\leq -\alpha h(x)||\nabla f(x)||^2$. 

(ii) For every $x_0\in \mathbb{R}^k$, the sequence $x_{n+1}=H(x_n)$ either satisfies $\lim _{n\rightarrow\infty}||x_{n+1}-x_n||=0$ or $
\lim _{n\rightarrow\infty}||x_n||=\infty$. Each cluster point of $\{x_n\}$ is a critical point of $f$. If moreover $f$ has at most countably many critical points, then $\{x_n\}$ either converges to a critical point of $f$ or $\lim _{n\rightarrow\infty}||x_n||=\infty$. 

(iii) There is a set $\mathcal{E}_1\subset \mathbb{R}^k$ of Lebesgue measure $0$ so that for all $x_0\in \mathbb{R}^k\backslash \mathcal{E}_1$,  the sequence $x_{n+1}=H(x_n)$, {\bf if converges}, cannot converge to a {\bf generalised} saddle point.  

(iv) There is a set $\mathcal{E}_2\subset \mathbb{R}^k$ of Lebesgue measure $0$ so that for all $x_0\in \mathbb{R}^k\backslash \mathcal{E}_2$, any cluster point of the sequence $x_{n+1}=H(x_n)$ is not a saddle point, and more generally cannot be an isolated generalised saddle point. 

When the local Lipschitz constants for $\nabla f$ are bounded from above by a continuous function, we use the same idea to prove the same result for a {\bf new discrete version} of  Backtracking GD. The condition we need is still more general than that required in those results mentioned above by Lee et al. and Panageas and Piliouras. Similar results hold for Backtracking versions of Momentum and NAG, first defined in our joint work with T. H. Nguyen.    

Since the literature on convergence for Gradient Descent methods can be very confusing, in the Appendix of this paper we will also provide a brief overview of previous major convergence results for Gradient Descent methods for the readers' convenience.  
\end{abstract}
\maketitle

\section{Introduction} This paper is about convergence to minima for Gradient Descent (GD) methods, with a view towards applications in Deep Neural Networks. In this section we explain briefly GD methods and state the main results of the paper, together with some remarks. A review the current state-of-the-art of theoretical results on convergence for GD methods, together with arguments for why one should care about convergence of these methods and allow more general cost functions in applications in Deep Neural Networks, is given in the Appendix of the paper. 

\subsection{Gradient descent methods.}\label{SubsectionGD} Being able to minimise a $C^1$ function $f:\mathbb{R}^k\rightarrow \mathbb{R}$ is an important problem in both theory and applications. In practical applications, one does not hope to find closed form solutions to this minimisation problem, but instead switch to iterative methods. 

In this paper, we concentrate on Gradient Descent (GD) methods, which are used in many fields such as Deep Learning. The general version of this method, invented by Cauchy in 1847, is as follows. Let $\nabla f(x)$ be the gradient of $f$ at a point $x$, and $||\nabla f(x)||$ its Euclidean norm in $\mathbb{R}^k$.  We choose randomly a point $x_0\in \mathbb{R}^k$ and define a sequence
\begin{eqnarray*}
x_{n+1}=x_n-\delta (x_n) \nabla f(x_n),
\end{eqnarray*}
where $\delta (x_n)>0$ (learning rate), is appropriately chosen. We hope that the sequence $\{x_n\}$ will converge to a (global) minimum point of $f$. 

The simplest and most known version of GD is Standard GD, where we choose $\delta (x_n)=\delta _0$ for all $n$, here $\delta _0$ is a given positive number. Because of its simplicity, it has been used frequently in Deep Neural Networks and other applications.  Another basic version of GD is (discrete) Backtracking GD, which works as follows. We fix real numbers $\delta _0>0$ and $0<\alpha ,\beta <1$. We choose $\delta (x_n)$ to be the largest number  $\delta $ among the sequence $\{\beta ^m\delta _0:~m=0,1,2,\ldots\}$ satisfying the Amijo's condition:
\begin{eqnarray*}
f(x_n-\delta \nabla f(x_n))-f(x_n)\leq -\alpha \delta ||\nabla f(x_n)||^2.  
\end{eqnarray*}

There are also the inexact version of GD (see e.g. \cite{bertsekas, truong-nguyen}). The main results of this paper hold also for the inexact version of Backtracking GD, but to keep the paper concise in order to convey better the main ideas behind, we concentrate on this exact version only.  More complicated variants of the above two basic GD methods include: Momentum, NAG, Adam,  for Standard GD (see an overview in \cite{ruder}); and Two-way Backtracking GD, Backtracking Momentum, Backtracking NAG  for Backtracking GD (first defined in \cite{truong-nguyen}). There is also a stochastic version, denoted by SGD, which is usually used to justify the use of Standard GD in Deep Neural Networks. An overview of the asymptotic behaviour of GD methods is included in Subsection \ref{SubsectionOverview}. 

For later use, here we recall that a function $f$ is in class $C^{1,1}_L$, for a positive number $L<\infty$, if $\nabla f(x)$ is {\bf globally} Lipschitz continuous with Lipschitz constant $L$. We also recall that a point $x_{\infty}$ is a cluster point of a sequence $\{x_n\}$ if there is a subsequence $\{x_{n_j}\}$ which converges to $x_{\infty}$. 

{\bf Remark.} In Subsection \ref{SubsectionMainResults}, we will define a {\bf new discrete version} of Backtracking GD, for a class of cost functions including all $C^2$ functions and all $C^{1,1}_L$ functions.  

\subsection{Abundance of saddle points.} Besides minima, other common critical points for a function are maxima and saddle points. In fact, for a $C^2$ cost function, a non-degenerate critical point can only be one of these three types. While maxima are rarely a problem for descent methods, saddle points can theoretically be problematic, as we will present later in this subsection. Before then, we recall definitions of saddle points and generalised saddle points for the sake of unambiguous presentation. Let $f:\mathbb{R}^k\rightarrow \mathbb{R}$ be a $C^1$ function. Let $x_0$ be a critical point of $f$ near it  $f$ is $C^2$. 

{\bf Saddle point.} We say that $x_0$ is a saddle point if the Hessian $\nabla ^2f(x_0)$ is non-singular and has both positive and negative eigenvalues. 

{\bf Generalised saddle point.} We say that $x_0$ is a {\bf generalised} saddle point if the Hessian $\nabla ^2f(x_0)$ has at least one negative eigenvalue. Hence, this is the case for a non-degenerate maximum point.  

In practical applications, we would like the sequence $\{x_n\}$ to converge to a minimum point.  It has been shown in \cite{dauphin-pascanu-gulcehre-cho-ganguli-bengjo} via experiments that for cost functions appearing in DNN the ratio between minima and other types of critical points becomes exponentially small when the dimension $k$ increases, which illustrates a theoretical result for generic functions  \cite{bray-dean}. Which leads to the question: Would in most cases GD  converge to a minimum? This question will be addressed in our main theorems, and readers can consult Subsection \ref{SubsectionOverview}  for a detailed discussion about previous work. Before the work in  \cite{lee-simchowitz-jordan-recht}, we are not aware of any work in the literature which systematically and rigorously treats the avoidance of saddle points under general settings, both for GD and other iterative methods such as Newton's.

\subsection{The main results}\label{SubsectionMainResults} We are now ready to state the main results of this paper. 
\begin{theorem}
Let $f:\mathbb{R}^k\rightarrow \mathbb{R}$ be a $C^{1}$ function, so that $\nabla f$ is locally Lipschitz continuous. Assume moreover that $f$ is $C^2$ near its generalised saddle points. Fix real numbers $\delta _0>0$ and $0<\alpha <1$. Then there is a smooth function $h:\mathbb{R}^k\rightarrow (0,\delta _0]$ so that the map $H:\mathbb{R}^k\rightarrow \mathbb{R}^k$ defined by $H(x)=x-h(x)\nabla f(x)$ has the following property: 

(i) For all $x\in \mathbb{R}^k$, we have $f(H(x)))-f(x)\leq -\alpha h(x)||\nabla f(x)||^2$. 

(ii) For every $x_0\in \mathbb{R}^k$, the sequence $x_{n+1}=H(x_n)$ either satisfies $\lim _{n\rightarrow\infty}||x_{n+1}-x_n||=0$ or $
\lim _{n\rightarrow\infty}||x_n||=\infty$. Each cluster point of $\{x_n\}$ is a critical point of $f$. If moreover, $f$ has at most countably many critical points, then $\{x_n\}$ either converges to a critical point of $f$ or $\lim _{n\rightarrow\infty}||x_n||=\infty$. 

(iii) There is a set $\mathcal{E}_1\subset \mathbb{R}^k$ of Lebesgue measure $0$ so that for all $x_0\in \mathbb{R}^k\backslash \mathcal{E}_1$,  the sequence $x_{n+1}=H(x_n)$, {\bf if converges}, cannot converge to a {\bf generalised} saddle point.  

(iv) There is a set $\mathcal{E}_2\subset \mathbb{R}^k$ of Lebesgue measure $0$ so that for all $x_0\in \mathbb{R}^k\backslash \mathcal{E}_2$, any cluster point of the sequence $x_{n+1}=H(x_n)$ is not a saddle point, and more generally cannot be an isolated generalised saddle point. 
\label{TheoremMain}\end{theorem}

Using the idea in the proof of Theorem \ref{TheoremMain}, we can prove the same conclusion for a  {\bf new discrete version of} Backtracking GD, under assumptions more general than those needed in \cite{lee-simchowitz-jordan-recht, panageas-piliouras} for Standard GD. We first describe this new discrete version of Backtracking GD.

\begin{definition}
(Backtracking GD-New.) Let $f:\mathbb{R}^k\rightarrow \mathbb{R}$ be a $C^1$ function. Assume that there are {\bf continuous} functions $r,L:\mathbb{R}^k\rightarrow (0,\infty)$ so that for each $x\in \mathbb{R}^k$, the map $\nabla f$ is Lipschitz continuous on $B(x,r(x))$ with Lipschitz constant $L(x)$. 

The Backtracking GD-New procedure is defined as follows. Fix $\delta _0>0$ and $0<\alpha , \beta <1$.  For each $x\in \mathbb{R}^k$, we define $\widehat{\delta} (x)$ to be the largest number $\delta$ among  $\{\beta ^n\delta _0:~n=0,1,2,\ldots \}$ which satisfies the two conditions
\begin{eqnarray*}
\delta &<&\alpha /L(x),\\
\delta ||\nabla f(x)|| &<& r(x).  
\end{eqnarray*}

For any $x_0\in \mathbb{R}^k$, we then define the sequence $\{x_n\}$ as follows
\begin{eqnarray*}
x_{n+1}=x_n-\widehat{\delta} (x_n)\nabla f(x_n). 
\end{eqnarray*}
\label{DefinitionBacktrackingGDNew}\end{definition}

{\bf Examples.}  (i) If $f$ is in $C^{1,1}_L$, then $f$ satisfies the condition in Definition \ref{DefinitionBacktrackingGDNew} by defining $L(x)=L$ for all $x$.  (ii) If $f$ is in $C^2$, then we can choose any continuous function $L(x)$ so that $L(x)\geq \max _{z\in B(x,r(x))}||\nabla ^2f||$, where $r:\mathbb{R}^k\rightarrow (0,\infty )$ is any continuous function. 

We have the following result for the new discrete version of Backtracking GD in Definition \ref{DefinitionBacktrackingGDNew}. 
  
\begin{theorem}
Let $f:\mathbb{R}^k\rightarrow \mathbb{R}$ be a $C^{1}$ function which satisfies the condition in Definition \ref{DefinitionBacktrackingGDNew}. Assume moreover that $\nabla f$ is $C^2$ near its generalised saddle points. Choose $0<\delta _0$ and $0<\alpha ,\beta <1$. For any $x_0\in \mathbb{R}^k$, we construct the sequence $x_{n+1}=x_n-\widehat{\delta} (x_n) \nabla f(x_n)$ as in Definition \ref{DefinitionBacktrackingGDNew}. Then: 

(i) For all $n$ we have $f(x_n-\widehat{\delta} (x_n)\nabla f(x_n))-f(x_n)\leq -(1-\alpha ) \widehat{\delta} (x_n)||\nabla f(x_n)||^2$. 

(ii) For every $x_0\in \mathbb{R}^k$, the sequence $\{x_{n}\}$ either satisfies $\lim _{n\rightarrow\infty}||x_{n+1}-x_n||=0$ or $
\lim _{n\rightarrow\infty}||x_n||=\infty$. Each cluster point of $\{x_n\}$ is a critical point of $f$. If moreover, $f$ has at most countably many critical points, then $\{x_n\}$ either converges to a critical point of $f$ or $\lim _{n\rightarrow\infty}||x_n||=\infty$. 

(iii) For {\bf random choices} of $\delta _0, \alpha $ and $\beta$, there is a set $\mathcal{E}_1\subset \mathbb{R}^k$ of Lebesgue measure $0$ so that for all $x_0\in \mathbb{R}^k\backslash \mathcal{E}_1$, any cluster point of the sequence $\{x_{n}\}$ cannot be a saddle point, and more generally cannot be an isolated generalised saddle point. 

(iv) Assume that there is $L>0$ so that if $x$ is a {\bf non-isolated} generalised saddle point of $f$, then $L(x)\leq L$. Then, for {\bf random choices} of $\delta _0, \alpha $ and $\beta$ with $\delta _0< \alpha /L $ or $\beta ^{n_0+1}\delta _0<\alpha /L < \beta ^{n_0}\delta _0$ for some $n_0$, there is a set $\mathcal{E}_2\subset \mathbb{R}^k$ of Lebesgue measure $0$ so that for all $x_0\in \mathbb{R}^k\backslash \mathcal{E}_2$, if the sequence $\{x_n\}$ converges, then the limit point cannot be a generalised saddle point. 
\label{Theorem1}\end{theorem}

{\bf Remarks.} Since the learning rates in both theorems are bounded by local Lipschitz constants of $\nabla f$,  if the sequence $\{x_n\}$ converges, then the convergence rate will satisfy the usual estimates in \cite{armijo}. One non-trivial situation when the condition in (iv) of Theorem \ref{Theorem1} is when we know by some reason that the set of all non-isolated critical points of $f$ is bounded. 
  
The main idea for the proof of Theorem \ref{TheoremMain} is as follows. From the assumption that $\nabla f$ is locally Lipschitz, we can choose locally in small open sets $U$ small enough learning rates $\delta$ so that  Armijo's condition is satisfied, $x\mapsto x-\delta \nabla f(x)$ is injective in that neighbourhood $U$ and moreover is a local diffeomorphism near generalised saddle points. Then we use a partition of unity to carefully glue together these local learning rates into a smooth positive, bounded function $h:\mathbb{R}^k\rightarrow (0,\delta _0]$. Then (i) holds by construction. In (ii), for showing that any cluster point of $\{x_n\}$ is a critical point of $f$, we use the arguments in \cite{bertsekas}.  For showing the remaining assertions in (ii), we follow the proofs in \cite{truong-nguyen} of the corresponding assertions for Backtracking GD. From the proof, we see that for parts (i) and (ii) only, we {\bf do not need} the assumption that $\nabla f$ is $C^2$ near its generalised saddle points. For proof of (iii), we use arguments in \cite{lee-simchowitz-jordan-recht, panageas-piliouras}. For (iv), we use additionally the fact that a saddle point is an isolated critical point of $f$, and the result in \cite{truong-nguyen} that the set of cluster points of $\{x_n\}$, considered in the real projective space $\mathbb{P}^k$, is connected.    

Note that parts (iii) and (iv) of Theorems \ref{TheoremMain} and \ref{Theorem1} have the same conclusion as the main results in \cite{lee-simchowitz-jordan-recht, panageas-piliouras}, while here we do not require that $\nabla f$ is globally Lipschitz continuous as in  those papers. While the learning rates in Theorem \ref{TheoremMain} (the function $h(x)$) here are not very explicitly determined, provided we can have an explicit estimate for the local Lipschitz constants of $\nabla f(x)$ then we can make $h(x)$ explicit. See examples in the next section for details, where we show that a similar construction can also be done under careful analysis for some maps whose gradient $\nabla f$ need not be locally Lipschitz. Note that in part (iv) of Theorem \ref{TheoremMain} and part (iii) of Theorem \ref{Theorem1}, when the cluster set of $\{x_n\}$ contains an isolated saddle point, a priori we not have that $\{x_n\}$ must  converge, but this fact turns out to follow from results in \cite{truong-nguyen}.  For a more detailed overview of previous results, the readers are invited to consult Subsection \ref{SubsectionOverview}. We note that (iii) and (iv) of Theorems \ref{TheoremMain} and \ref{Theorem1} answer in affirmative some variants of, and hence give support to, the following conjecture in \cite{truong-nguyen} (here stated a bit stronger than the original version, in view of the results we prove in this paper): 

\begin{conjecture} (Conjecture 5.1 in \cite{truong-nguyen}) Let $f:\mathbb{R}^k\rightarrow \mathbb{R}$ be a $C^1$ function and $C^2$ near its generalised saddle points. Then the set of initial points $x_0\in \mathbb{R}^k$ for which the cluster points of the sequence $\{x_n\}$ - constructed by the Backtracking GD method - contains a generalised saddle point has Lebesgue measure zero. 
 \label{ConjectureSaddlePoint}\end{conjecture}
 
 \subsection{Plan of the paper} In Section 2 we prove Theorems \ref{TheoremMain} and \ref{Theorem1} and provide some examples. In Section 3 we give a summary and discuss future research directions. In the Appendix we  discuss why convergence in GD is important in practice, why we should allow cost functions $f$ as general as possible, and give an overview of previous work on convergence for GD methods to help reduce readers' confusion with the state-of-the-art in this subject. 

\section{Proofs of the main results and Some examples} In this section, first we prove Theorems \ref{TheoremMain} and \ref{Theorem1}. After that, we give some examples applying the theorems. For the proof, we will make use of the following two simple results. The first of which was used since \cite{armijo} in GD methods, and the second is a simple estimate on the Lebesgue measure of the preimage of a map whose distortion is bounded away from $0$. 

\begin{lemma}
Let $U\subset \mathbb{R}^k$ be an open and convex set, and let $f:U\rightarrow \mathbb{R}$ be a $C^1$-function, whose gradient $\nabla f$ is Lipschitz continuous on $U$ with Lipschitz constant $L$. Let $x_0\in U$ and $\delta >0$ so that $x_0-\delta \nabla f(x_0)\in U$. Then 
\begin{eqnarray*}
f(x_0-\delta \nabla f(x_0))-f(x_0)\leq -\delta (1-\delta /L) ||\nabla f(x_0)||^2. 
\end{eqnarray*}
\label{LemmaSimpleEstimate}\end{lemma}
\begin{proof}
By the Fundamental Theorem of Calculus, we have
\begin{eqnarray*}
f(x_0-\delta \nabla f(x_0))-f(x_0)=-\delta \int _0^1\nabla f (x_0-s\delta \nabla f(x_0)).\nabla f(x_0)ds. 
\end{eqnarray*}
Plugging into the RHS the following estimate
\begin{eqnarray*}
\nabla f (x_0-s\delta \nabla f(x_0)).\nabla f(x_0)&=&||\nabla f(x_0)||^2+(\nabla f (x_0-s\delta \nabla f(x_0)-\nabla f(x_0)).\nabla f(x_0))\\
&\geq & ||\nabla f(x_0)||^2-s\delta L||\nabla f(x_0)||^2,
\end{eqnarray*}
we obtain the result. 
\end{proof}

\begin{lemma}
Let $U\subset \mathbb{R}^k$ be an open and convex subset, and $H:U\rightarrow \mathbb{R}^k$ a continuous function. Assume that there is $\lambda >0$ so that $||H(x)-H(y)||\geq \lambda ||x-y||$ for all $x,y\in U$. Let $\mathcal{E}\subset \mathbb{R}^k$ be a measurable set of Lebesgue measure $0$. Then $H^{-1}(\mathcal{E})$ is also of Lebesgue measure $0$.
\label{LemmaMeasureEstimate}\end{lemma}
\begin{proof}
Since $\mathcal{E}$ has Lebesgue measure $0$, for each $\epsilon >0$ there is a sequence of balls $\{B(x_i,r_i)\}_{i=1,2,\ldots }$ so that $\mathcal{E}\subset \bigcup _{i=1}^{\infty}B(x_i,r_i)$ and $\sum _i(Vol(B(x_i,r_i)))\sim \sum _ir_i^k<\epsilon$. 

Since $||H(x)-H(y)||\geq \lambda ||x-y||$ for all $x,y\in U$, we get that $H^{-1}(\mathcal{E})$ $\subset$  $\bigcup _iB(H^{-1}(x_i),r_i/\lambda )$ and
\begin{eqnarray*}
\sum _iVol(B(H^{-1}(x_i),r_i/\delta ))\sim \sum _i (r_i/\delta )^k <\epsilon /\lambda ^k.
\end{eqnarray*} 
Therefore, the Lebesgue measure of $H^{-1}(\mathcal{E})$ is $<\epsilon /\lambda ^k$ for all $\epsilon >0$, and hence must be $0$. 
\end{proof}

\subsection{Proof of Theorem \ref{TheoremMain}} Since $\nabla f$ is locally Lipschitz continuous, for each $x\in \mathbb{R}^k$, there are positive numbers $r(x),L(x)>0$ so that  $\nabla f(x)$ has Lipschitz constant $L(x)$ in the ball $B(x,r(x))$. That is, for all $y,z\in B(x,r(x))$ we have $||\nabla f(y)-\nabla f(z)||\leq L(x)||y-z||$. Also, we can choose $L(x)$ large enough so that for all $0<\delta \leq 1/(L(x))$ then $f(z-\delta \nabla f(z))-f(z)\leq -\alpha \delta ||\nabla f(z)||^2$ for all $z\in B(x,r(x))$.  

 There is a partition of unity $\{\varphi _j\}_{j=1,2,\ldots }$ of $\mathbb{R}^k$ with compact supports so that for every $j\in \mathbb{N}$, there is a point $z_j\in \mathbb{R}^k$ for which the support $supp(\varphi _j)$ of $\varphi _j$ is contained in $B(z_j,r(z_j))$. Moreover, $\{supp (\varphi _j)\}_j$  is locally finite, that is  every $x\in \mathbb{R}^k$ has an open neighbourhood $U$ which intersects only a finite number of those $supp (\varphi _j)$'s. (This is related to Lindel\"off theorem, mentioned in the appendix of this paper. We recall here the main idea: Any open cover of $\mathbb{R}^k$ has a subcover which is locally finite. We can even arrange that each point $x\in \mathbb{R}^k$ is contained in at most $k+1$ open sets in the subcover. Then we construct a partition of unity with compact supports contained in open sets in that subcover.). For each $j=1,2,\ldots $, we let 
\begin{eqnarray*}
M_j=\max _{y_1,y_2\in B(z_j,r(z_j))}||\nabla \varphi _j(y_1)|| \times ||\nabla f(y_2)||. 
\end{eqnarray*}  
  
 We now define the function $h:\mathbb{R}^k\rightarrow \mathbb{R}$ by the following formula: 
\begin{eqnarray*}
h(x):=\sum _{j=1}^{\infty} \frac{1}{10^{j}(M_j+1)} \varphi _j(x)\min \{\frac{1}{2L(z_j)},\delta _0,1\}.
\end{eqnarray*}

Since $\varphi _j$'s are non-negative, smooth and their supports are locally finite, it follows that the function $h(x)$ is well-defined, smooth and non-negative. Since $\sum _j\varphi _j(x)=1$, it follows that $0<h(x)\leq \delta _0$ for all $x\in \mathbb{R}^k$. 

Fix a point $x\in \mathbb{R}^k$. Then there is at least one $j$ so that $\varphi _j(x)>0$, and hence $h(x)\geq \frac{1}{10^j(M_j+1)} \varphi _j(x)\min \{\frac{1}{2L(z_j)},\delta _0,1\} >0$.  Then there is a finite index set $J$ and an open neighbourhood $V$ of $x$ so that if $supp (\varphi _j)\cap V\not= \emptyset$ then $j\in J$. Since $supp(\varphi _j)$'s are all compact, we can shrink $V$ so that $x\in \bigcap _{i\in J}B(z_j,r(z_j))$ and $h(x)\leq \max _{j\in J}\min \{\frac{1}{2L(z_j)},\delta _0\}$. Since $h$ is a smooth function, we can find an open neighbourhood $U$ of $x$ so that $U\subset  V\cap \bigcap _{i\in J}B(z_j,r(z_j))$ and for all $y\in U$ we have $h(y)\leq  \max _{j\in J}\min \{\frac{2}{3L(z_j)},\delta _0\}$. It then follows by the choice of $L(z_j)$'s that for all $y\in U$ we have
\begin{eqnarray*}
f(y-h(y)\nabla f(y))-f(y)\leq -\alpha h(y) ||\nabla f(y)||^2. 
\end{eqnarray*}   
We will now show also that $H(y)=y-h(y)\nabla f(y)$ is injective on this same set $U$. In fact, assume that there are distinct $y_1,y_2\in U$ so that  $y_1-h(y_1)\nabla f(y_1)=y_2-h(y_2)\nabla f(y_2)$. We rewrite this as: 
\begin{eqnarray*}
(y_1-y_2)-h(y_1)(\nabla f(y_1)-\nabla f(y_2))-(h(y_1)-h(y_2))\nabla f(y_2)=0. 
\end{eqnarray*}
Since $U\subset V$ and by the choice of $V$, if $\varphi _j(y)\not= 0$ for some $y\in V$ then $j\in J$.  Hence, we obtain a contradiction because
\begin{eqnarray*}
||h(y_1)(\nabla f(y_1)-\nabla f(y_2))||\leq \frac{2}{3}||y_1-y_2||,
\end{eqnarray*}
and using that $|\varphi _j(y_1)-\varphi _j(y_2)|\leq ||y_1-y_2|| \times \max _{B(z_j,r(z_j))}||\nabla \varphi _j||$
\begin{eqnarray*}
&&||(h(y_1)-h(y_2))\nabla f(y_2)||\\
&=&|\sum _{j\in J}  \frac{1}{10^{j}(M_j+1)} (\varphi _j(y_1)-\varphi _j(y_2))\min \{\frac{1}{2L(z_j)},\delta _0,1\}|\times ||\nabla f(y_2)||\\
&\leq&\sum _{j\in J}  \frac{1}{10^{j}(M_j+1)} M_j\min \{\frac{1}{2L(z_j)},\delta _0,1\}||y_1-y_2||\\
&\leq& \frac{1}{9}||y_1-y_2||. 
\end{eqnarray*}

Near generalised saddle points of  $f$, the map $x\mapsto H(x)$ is moreover $C^1$ by the assumption that $f$ is $C^2$ near those points. Hence, $x\mapsto H(x)$ is a local diffeomorphism near generalised saddle points of $f$. 

{\bf Proof of (i)}: Already given above. 

{\bf Proof of (ii)}: Since $h(x)$ is smooth and positive, it follows that for every compact subset $K\subset \mathbb{R}^k$ where $\inf _{x\in K}||\nabla f||>0$, we have $\inf _{x\in K}h(x)>0$. Hence, we can use the same arguments as in \cite{bertsekas} to show that any cluster point of the sequence $\{x_{n+1}=H(x_n)\}$ is a critical point of $f$. 

Since $h(x)\leq \delta _0$ for all $x\in \mathbb{R}^k$, we can use the same proof of part 1 of Theorem 2.1 in \cite{truong-nguyen} to conclude that either $\lim _{n\rightarrow\infty}||x_{n+1}-x_n||=0$ or $\lim _{n\rightarrow\infty}||x_n||=\infty$. 

Then we can use the same proof of part 2 of Theorem 2.1 in \cite{truong-nguyen}, by employing the real projective space $\mathbb{P}^k$ and result in \cite{asic-adamovic} for cluster points of sequences $\{x_n\}$ in compact metric space $(X,d)$ satisfying $\lim _{n\rightarrow\infty}d(x_{n+1},x_n)=0$, to show that if $f$ has at most countably many critical points, then either $\{x_n\}$ converges to a critical point of $f$ or $\lim _{n\rightarrow\infty}||x_n||=\infty$. 

{\bf Proof of (iii)}: By using Stable-Center Manifold theorem for local diffeomorphisms and using that the map $x\mapsto H(x)$ is a local diffeomorphism near generalised saddle points, we can argue as in \cite{lee-simchowitz-jordan-recht, panageas-piliouras} that there is an open neighbourhood $U$ of generalised saddle points of $f$, and a subset $\mathcal{F}_1\subset U$ of Lebesgue measure $0$, so that if all $\{x_n\}\subset U\backslash \mathcal{F}_1$ and $\{x_n\}$ {\bf converges}, then the limit point cannot be a saddle point.   

Since we showed in the construction that the map $x\mapsto H(x)$ is locally injective, it follows that $\mathcal{E}_1=\bigcup _{n\in \mathbb{N}}H^{-n}(\mathcal{F}_1)$ also has Lebesgue measure $0$. Then it follows that if $x_0\in \mathbb{R}^k\backslash \mathcal{E}_1$, then $\{x_n\}$ cannot converge to a generalised saddle point. 

{\bf Proof of (iv)}: We use the ideas in \cite{truong-nguyen}. Note that a saddle point of $f$ is a non-degenerate critical point, and hence is an isolated generalised saddle point. Assume that the set of cluster points $A$ of $\{x_n\}$ contains an isolated generalised saddle point $y_0$. Then the property $\lim _{n\rightarrow\infty}||x_n||=\infty$ does not hold, and hence by part (ii) we must have $\lim _{n\rightarrow\infty}||x_{n+1}-x_n||=0$. Then, by the result in \cite{asic-adamovic} for the real projective space $\mathbb{P}^k$, it follows that the closure $\overline{A}\subset \mathbb{P}^k$ is connected. By part (ii) again, the set A is contained in the set of all critical points of $f$, and hence $y_0$ is also an isolated point of A, and hence of $\overline{A}$. Thus $A=\{y_0\}$, and hence $\lim _{n\rightarrow\infty}x_n=y_0$. Then we can use part (iii) to conclude.   

\subsection{Proof of Theorem \ref{Theorem1}} Note that by construction we have $\widehat{\delta} (x_n)\leq \delta _0$ for all $n$. Since $L(x), r(x)$ and $\nabla f(x)$ are continuous in $x$, it follows that for all compact subset $K\subset \mathbb{R}^k$, we have $\inf _{x\in K}\delta (x)>0$. Therefore, (i) and (ii) of Theorem \ref{Theorem1} follows from  \cite{truong-nguyen}, note that by construction $x-\widehat{\delta}(x)\nabla f(x)\in B(x,r(x)/2)$ for all $x\in \mathbb{R}^k$ and hence Armijo's condition is satisfied. 

Now we prove (iii). Let $\mathcal{S}$ be the set of {\bf isolated} generalised saddle points of $f$. Then $\mathcal{S}$ is countable. Therefore, for a {\bf random choice} of $\alpha, \beta, \delta _0$, we have the following: for every $x\in \mathcal{S}$ then $\alpha /L(x)$ does not belong to $\{\beta ^n\delta _0: ~n=0,1,2,\ldots \}$, and also $r(x)/||\nabla f(x)||$ does not belong to $\{\beta ^n\delta _0: ~n=0,1,2,\ldots \}$. 

Hence, for each $x\in \mathcal{S}$, either $\alpha /L(x)>\delta _0$, or there is a number $n(x_0)$ so that
\begin{eqnarray*}
\beta ^{n(x_0)+1}\delta _0<\frac{\alpha }{L(x)}< \beta ^{n(x_0)}\delta _0. 
\end{eqnarray*}
In both cases, since $z\mapsto L(z)$ is continuous, there is an open neighbourhood $U(x)$ of $x$ so that for all $z\in U(x)$, then $\alpha /L(z)$ has the same behaviour.  Shrinking $U(x)$ if necessary,  we can assure that $||\nabla f(z)||$ is small in $U(x)$, and hence $r(z)/||\nabla f(z)||>\delta _0$.Therefore, by Definition \ref{DefinitionBacktrackingGDNew}, we have that $\widehat{\delta} (z)=\widehat{\delta} (x)=\beta ^{n(x_0)+1}\delta _0$ for all $z\in U(x)$. In particular, the map $z\mapsto H(z)=z-\widehat{\delta} (z)\nabla f(z)$ is a local diffeomorphism in $U(x)$. 

By \cite{truong-nguyen}, if the cluster set of $\{x_n\}$ contains an isolated generalised saddle point, then $\{x_n\}$ converges to that generalised saddle point. Then, we can apply Stable-Central theorem to obtain that there is an open neighbourhood $U$ of $\mathcal{S}$ and a set $\mathcal{F}\subset U$ of Lebesgue measure $0$ such that if $x_0\in \mathbb{R}^k$ is so that the cluster points of $\{x_n\}$ contains an isolated generalised saddle point, then there is $n(x_0)$ for which $H^{n(x_0)}(x_0)\in \mathcal{F}$. 

We note that since $z\mapsto L(z)$, $z\mapsto r(z)$ and $z\mapsto ||\nabla f(z)||$ are continuous, for every $x\in \mathbb{R}^k$, there is a neighbourhood $U(x)$ so that for all $z\in U(x)$ then $\widehat{\delta} (z)=\widehat{\delta} (x)$ or $\widehat{\delta} (z)=\widehat{\delta} (x)/\beta$. Then we see that $z\mapsto H(z)$ is injective on both sets $\{z\in U(x): ~\widehat{\delta} (z)=\widehat{\delta} (x)\}$ and $\{z\in U(x): ~\widehat{\delta} (z)=\widehat{\delta} (x)/\beta \}$. Then argue as in the proof of (iii) and (iv) in Theorem \ref{TheoremMain}, we get that the set 
\begin{eqnarray*}
\mathcal{E}=\bigcup _{n\in \mathbb{N}}H^{-n}(\mathcal{F})
\end{eqnarray*}
has Lebesgue measure $0$, and for all $x_0\in \mathbb{R}^k\backslash \mathcal{E}$, the sequence $\{x_n\}$ constructed from Definition \ref{DefinitionBacktrackingGDNew} cannot have any cluster point which is an isolated generalised saddle point. 

The proof of (iv) is similar. Here, the assumption that $\alpha /L>\delta _0$ or $\beta ^{n_0+1}\delta _0<\alpha /L <\beta ^{n_0}$ and that $L\geq L(x)$ for all {\bf non-isolated}  generalised saddle points imply that for $z$ near a non-isolated saddle point we have $\delta (z)=\delta (x)$ (note that since $x$ is a critical point we have $\nabla f=0$, and hence for $z$ near $x$ we always have $\delta _0||\nabla f(z)||<r(z)$), and hence the map $z\mapsto H(z)$ is a local diffeomorphism near those generalised saddle points as well.  
 
\subsection{Some examples}\label{SubsectionExamples} Looking at the proofs of Theorems \ref{TheoremMain} and \ref{Theorem1}, we see that if we can estimate $L(x)$ and $r(x)$ more exactly, then we can construct more explicitly the function $h(x)$. Below we illustrate this with a specific function.

{\bf Example 1}: $f:\mathbb{R}^2\rightarrow \mathbb{R}$ given by $f(x,y)=x^3\sin (1/x) + y^3\sin (1/y)$. It can be checked that $f$ is $C^1$ on $\mathbb{R}^2$, $f$ is $C^2$ on $\mathbb{R}^2\backslash (\{x=0\}\cup \{y=0\})$,  but $\nabla f$ is not locally Lipschitz at any point on $\{x=0\}\cup \{y=0\}$. Moreover, $f$ is neither convex nor real analytic. Since the function $p(t)=t\mapsto t^3\sin(1/t)$ satisfies $\lim _{t\rightarrow\pm \infty}p(t)=+\infty$, the function $f(x,y)$ has compact sublevels. 

First, we apply Theorem \ref{TheoremMain}. 

In this example, when both $x,y\not= 0$, we can compute 
\begin{eqnarray*}
\nabla f(x,y) = (3x^2\sin (1/x)-x\cos (1/x), 3y^2 \sin (1/y)-y\cos (1/y)). 
\end{eqnarray*}
From this, by using properties of analytic functions,  we can see that $f$ has countably many critical points, and countably many of them are saddle points. 

Computing the Hessian, and using the rough estimates $|\cos (z)|, |\sin (z)|\leq 1$, we obtain a rough estimate for the Hessian at points $(x,y)\notin \{x=0\}\cup \{y=0\})$:
\begin{eqnarray*}
||\nabla ^2f (x,y)||\leq 6(|x|+|y|)+8+(\frac{1}{|x|}+\frac{1}{|y|}). 
\end{eqnarray*}
Note that $\sup _{(x,y)\notin \{x=0\}\cup \{y=0\})}||\nabla ^2f(x,y)||=\infty$, and hence the gradient $\nabla f$ is not globally Lipschitz continuous.    
 
Hence, we can explicitly construct, as in the proof of Theorem \ref{TheoremMain}, a smooth function $h:\mathbb{R}^2\backslash (\{x=0\}\cup \{y=0\})\rightarrow (0,\delta _0]$ which is locally injective and a local diffeomorphism near generalised saddle points in the domain $\mathbb{R}^2\backslash (\{x=0\}\cup \{y=0\})$. Moreover, Armijo's condition is satisfied. 

Near the points in $\{x=0\}\cup \{y=0\}$, since $\nabla f$ is not locally Lipschitz, we cannot do the same construction as in the proof of Theorem \ref{TheoremMain}. However, we observe that if $x=0$ then $\partial f/\partial x$ is also $0$. Hence, for points in $\{x=0\}\backslash \{(0,0)\}$, we construct $h(0,y)$ for {\bf the restriction} of the function $f$ to $\{x=0\}$. Similarly, for points in $\{y=0\}\backslash \{(0,0)\}$ we construct $h(x,0)$ for {\bf the restriction} of the function $f$ to $\{y=0\}$. The only point left where we need to construct $h$ is $(0,0)$, but since this point is a critical point of $f$, we can just choose any value for $h(0,0)$. Alternatively, on points in $\{x=0\}\cup \{y=0\}$, we just define $h(x,y)$ using the usual Backtracking GD procedure.

Note that the function $h:\mathbb{R}^2\rightarrow \mathbb{R}$ we construct here is {\bf not even continuous}. However, we can use the same ideas as in \cite{truong-nguyen} to prove the conclusions (i) and (ii) of Theorem \ref{TheoremMain} for the map $H(x,y): ~(x,y)\mapsto (x,y)-h(x,y)\nabla f(x,y)$. Moreover, we have that on $\mathbb{R}^2\backslash (\{x=0\}\cup \{y=0\})$, the map $H(x,y)$ is locally injective and is a local diffeomorphism near generalised saddle points of $f$. Note that by definition, generalised saddle points of $f$ can only be contained in $\mathbb{R}^2\backslash (\{x=0\}\cup \{y=0\})$. Similarly, the restriction $H(0,y)$ is locally injective on $\{x=0\}\backslash \{(0,0)\}$, and the restriction $H(x,0)$ is locally injective on  $\{y=0\}\backslash \{(0,0)\}$. Moreover, the preimage of any point in  $\{x=0\}\cup \{y=0\}$ is at most countable. From this, we obtain that if $\mathcal{E}\subset \mathbb{R}^2$ is of Lebesgue measure $0$, then $\bigcup _{n\in \mathbb{N}}H^{-n}(\mathcal{E})$ also has Lebesgue measure $0$. In this case, except the point $(0,0)$, other critical points of $f$ are isolated. Hence, we can state the following stronger property than stated in (iii) and (iv) of Theorem \ref{TheoremMain}: 

Proposition: There exists a set $\mathcal{E}\subset \mathbb{R}^2$ of Lebesgue measure $0$, so that for all $(x_0,y_0)\in \mathbb{R}^2\backslash \mathcal{E}$, the sequence $(x_{n+1},y_{n+1})=H(x_n,y_n)$ satisfies the following: $(x_n,y_n)$ converges to a critical point of $f$ which is not a generalised saddle point.  

We can also obtain a similar Proposition by applying Theorem \ref{Theorem1}. We just need to make sure to choose $r(x,y)$ in the remark after Definition \ref{DefinitionBacktrackingGDNew} to be smaller than the distance from $(x,y)$ to $\{x=0\}\cup \{y=0\}$. 

Experiments show that the behaviour of the usual discrete Backtracking GD for this function is also similar to that in the above Proposition, and conforms to Conjecture \ref{ConjectureSaddlePoint}. 

{\bf Example 2:} We can explore similar examples such as $f(x,y)=a x^p\sin ^q(1/x) + b y^p\sin ^q(1/y)$. 

\section{Conclusions}

In Example 2.17 in  \cite{truong-nguyen}, it was shown that even for a smooth cost function $f$, the map $x\mapsto x-\delta (x)\nabla f(x)$, where $\delta (x)$ is constructed from (discrete) Backtracking GD, is not always continuous. Hence, we cannot apply the usual continuous Dynamical Systems theory to study the asymptotic behaviour of Backtracking GD as in the case of the Standard GD algorithm where learning rate is fixed. In this paper, we have shown that if $\nabla f$ is locally Lipschitz-  but need  not be globally Lipschitz - then a {\bf continuous}  Dynamical System can be associated to the continuous version of Backtracking GD. We can prove the same conclusions as in the main results in \cite{truong-nguyen}. In the case $f$ is $C^2$ near its generalised saddle points, we can prove the same conclusions as in the main results in \cite{lee-simchowitz-jordan-recht, panageas-piliouras} but under less restrictive assumptions. In fact, as far as we know, the assumptions in Theorem \ref{TheoremMain} are also the least restrictive and most practical among all current known results for convergence to minima for all iterative methods, for general non-convex cost functions.  The same results can be proven for the Inexact version of Backtracking GD, and also for Backtracking versions of Momentum and NAG  as defined in \cite{truong-nguyen}. If moreover $f$ satisfies the condition in Definition \ref{DefinitionBacktrackingGDNew}, then we can prove similar results for a {\bf new discrete version} of Backtracking GD. We remind that the condition in Definition \ref{DefinitionBacktrackingGDNew} is satisfied for cost functions which are either $C^{1,1}_L$ or $C^2$. We have illustrated in Examples 1 and 2 in Subsection \ref{SubsectionExamples} how the construction can be done in practice, for very singular functions.  

{\bf Future work.} There are some promising directions worth pursue in the future. 

First direction: Since the construction of $h(x)$ in Theorem \ref{TheoremMain} is rather implicit, it is interesting to see how to implement it in practice, in particular for complicated cost functions used in Deep Neural Networks. While the construction of $\widehat{\delta} (x)$ in Theorem \ref{Theorem1} is easy provided $L(x)$ and $r(x)$ are known, determining $L(x),r(x)$ effectively in practice (with complicated cost functions such as those used in DNN) can be challenging. The ideas presented in Subsection \ref{SubsectionExamples} may be helpful. 

Second direction: Extending the results in this paper and \cite{truong-nguyen} to more general functions and optimization on infinite dimensional vector spaces or 
other manifolds. It is also of great interest to extend the results to constraint optimization problems.  

Third direction: Using ideas from this paper, combined with Random Dynamical Systems, to solve Conjecture \ref{ConjectureSaddlePoint}.    

Fourth direction: Using (discrete or continuous) Backtracking GD to help resolve challenges such as adversarial attacks \cite{eykholt-etal, finlayson-etal, papernot-etal}, since learning rates in Backtracking GD are chosen very adaptively with respect to the data, and since so far the best convergence results for iterative optimisation methods are obtained for Backtracking GD under least restrictive conditions. 

\section{Appendix} In this section we first explain why convergence of the iteration process in GD and allowing more general cost functions are important from a practical view point. Then we present a brief overview of major convergence results in previous literature for GD methods and their use in Deep Neural Networks, to help dispel some widespread misunderstandings about this topic, which are apparent from lecture notes, books, videos and posts on professional websites, in both Optimization and Deep Learning communities. Given that Deep Learning is still not as reliable, reproducible and safe as desired (existence of adversarial images - both synthetic and physical and in medical imaging  \cite{eykholt-etal, finlayson-etal, papernot-etal}, as well as fatal accidents in self driving cars - as recent as September 2019), there is still a large demand for rigorously theoretical guarantees for the practices employed in Deep Learning - even though good experimental results are abundant and appear more and more.    

\subsection{Why is convergence of the iteration in GD important and Why should we allow general cost functions?}\label{SubsectionConditions}  Here we provide some reasons for why we should allow the cost function $f$ as most general as possible and the convergence of the sequence $\{x_n\}$ constructed in the previous paragraph is important, based on its use in Deep Learning.  

The main idea behind the use of Deep Neural Networks (DNNs) in Deep Learning is an approximation result for continuous functions, which we will briefly present here.  Assume for example that we want to have an AI to classify very well hand written digits. Then we assume that there is a correct assignment, for each image $z$, a classification $y(z)\in \{0,1,\ldots ,9\}$. Since it is impractical and impossible to employ humans to look at each and every possible image $x$ out there and assign the correct label $y(x)\in \{0,1,2,\ldots ,9\}$, we will only classify by hand a small subset $I$ (training set) of images, and then employ an appropriate DNN. Mathematically, a DNN is a finite composition of functions $h$ of the form $(\sigma \circ L_1,\ldots ,\sigma \circ L_j)$, where $L_1,\ldots ,L_j$ are affine functions from a finite dimensional vector space (whose dimensions can vary, and whose coefficients $\gamma$ are not given in advance but will be chosen later via an optimisation problem so that to best approximate the given data in the training set $I$ with respect to a choice of metric) to $\mathbb{R}$, and $\sigma :\mathbb{R}\rightarrow \mathbb{R}$ is an appropriate nonlinear function.  The justification of such a use of DNN is that any continuous function can be approximated by such DNN, see e.g.   \cite{pinkus} for details. In analogy to biology, each function $h$ appearing in a DNN corresponds to a layer of that DNN, and the dimension of the corresponding finite dimension vector space is called the number of neurons of the layer. Modern DNN can have hundreds of layers and millions of parameters. 

While currently some choices of the activation function $\sigma$ are favourited (such as the sigmoid function), we now argue that it is better to allow $\sigma$  as general as possible. In fact, because of limitations in resources and time, we cannot work with a DNN of arbitrary large number of neurons or layers. Hence, since the approximation theorem mentioned above is only valid when we allow the number of neurons or layers go to infinity, and since no finite model can work well for all questions (No-free-lunch-theorems, see \cite{wolpert-mcready}), it is wise that we work with as general DNNs as possible. 

As explained above, when we already chose a DNN for the question at hand, we are led to an optimisation problem of finding parameters $\gamma _{min}$ achieving minimum for a cost function $f(\gamma )$. Then such a $\gamma _{min}$ is used to provide predictions for new data outside of the training set $I$. In practice, we cannot find a closed form for such $\gamma _{min}$, and hence must make use of one of iterative methods, such as GD. In that case, we choose a random value $\gamma _0$ and then compute iteratively $\gamma _{n+1}=\gamma _n -
\delta (\gamma _n) \nabla f(\gamma _n)$. Even if we already chose on such numerical method, we cannot run the iteration infinitely many times, and so we need to stop after a finite time, say $n_0$, and then use $\gamma _{n_0}$ for new predictions.  Hence, we see here that convergence results for such iterative processes is important if we want the results to be stable and reproducible. This is because if another (or even oneself) will run the iteration again, even with the same choice of the initial point $\gamma _0$, there can be errors in computing which leads to another stop time $n_1$, and hence will use $\gamma _{n_1}$. The practice of using mini-batches in Deep Learning makes this scence even more realistic.  If no convergence result is guaranteed for the sequence $\{\gamma _n\}$, there is reason to doubt that the predictions one gets when using $\gamma _{n_0}$ and $\gamma _{n_1}$ are similar. This could lead to non-reproducibility and instability. 

While it is favourited in many previous and current works to require quite strong assumptions (see Subsection \ref{SubsectionOverview}) on the cost function $f$, such as being $C^{1,1}_L$, having compact sublevels, having just a finite number of critical points, being convex or real analytic, here we argue that even just from a practical view point (without taken into account the theoretical viewpoint), it is better to allow more general cost functions in the treatment. First, if one agrees with the previous paragraphs that the activation functions should be as general as possible, one sees that one needs to deal with general cost functions. Second, when one computes with functions in practice, errors are unavoidable, and hence one in principle must work with perturbations of the ideal cost functions, and perturbations in general do not preserve the mentioned properties. There is also another source where perturbations come from. To avoid convergence to bad minima, it is a common practice in the Deep Learning community to perturb the cost function by a term of the form $\lambda ||\gamma ||^p$, where $\lambda ,p>0$ are given numbers. If $p\not= 2$, then the resulting is not $C^{1,1}_L$. Even if $p=2$, the assumption about finiteness of critical points may not be preserved in general. Third, even when the cost function one starts with is in $C^{1,1}_L$, it may be difficult to have a good estimate of $L$ to make sure that one actually chooses correctly a learning rate $\delta$ which consents with theoretical assumptions. Hence, even in this good case, it is better if one can use a method which requires less checking or guaranteed to work under very general assumptions. Also, when new data comes, it would be tedious and difficult for one to change learning rates by hand. This has led to the current practice of manual fine-tuning of learning rates in Deep Learning, which - as we explained in the above paragraph - can lead to instability and unreproducibility, besides costing time and computer resources to redo with many different learning rates. 

All in all, we advocate for that it is better to use methods guaranteed for more general assumptions preserved under small perturbations.   

\subsection{Overview on previous major results on convergence and avoidance of saddle points.}\label{SubsectionOverview} Since GD has been developed in a long time and there are too many documents devoted to it, it may be an impression to many casual readers, or even experts, that the convergence of these methods have been established under very general assumptions and since a long time ago. This impression can be amplified when one reads or watches books, lecture notes, papers or videos, where there are many handwaving statements such as "Gradient Descent always converges to minimum", without the bother to give precise references and conditions under which proofs for such statements have been rigorously given. It may come as a surprise for one when finding that many such statements are groundless, and may contribute to wrong applications of the results or  false/unfair judgements/attributions about results.  To help dispel such misunderstandings, we collect here some previous major convergence results for GD so far with proper references for proofs.  

In the influential paper \cite{armijo}, where Armijo's condition was introduced into GD, Armijo proved the convergence of Standard GD and Backtracking GD for functions in $C^{1,1}_L$, under the further assumptions including that the function has only one critical point.  The most general result which can be proven with his method is (see e.g. Proposition 12.6.1 in \cite{lange}) the convergence of Standard GD under the assumption that $f$ is in $C^{1,1}_L$, the learning rate $\delta $ is $<1/(2L)$, $f$ has compact sublevels (that is, all the sets $\{x:~f(x)\leq b\}$ for $b\in \mathbb{R}$ are compact), and the set of critical points of $f$ is bounded and isolated (which in effect implies that $f$ has only a finite number of critical points). An analog of this result for gradient descent flow (solutions to $x'(t)=-\nabla f(x(t))$) is also known (see e.g. Appendix C.12 in \cite{helmke-moore}). Besides gradient flows, there are currently also work employing more techniques in PDE to solve optimization questions, however as far as we know assumptions needed are still similar to those mentioned in this paragraph. Both the assumptions that $f$ is in $C^{1,1}_L$ and $\delta $ is small enough are necessary for the conclusion of Proposition 12.6.1 in \cite{lange}, even for very simple functions, as shown by Examples 2.14 (for functions of the form $f(x)=|x|^{1+\gamma}$, where $0<\gamma <1$ is a rational number) and 2.15 (for functions which are smooth versions of $f(x)=|x|$) in \cite{truong-nguyen}. In contrast, for these examples, Backtracking GD always converges to the global minimum $0$.

Concerning the issue of saddle points, (\cite{lee-simchowitz-jordan-recht, panageas-piliouras}) proved the very strong result that for $f$ in $C^{1,1}_L$ and $\delta < 1/L$, there exists a set $\mathcal{E}\subset \mathbb{R}^k$ of Lebesgue measure $0$ so that for $x_0\in \mathbb{R}^k\backslash \mathcal{E}$, if the sequence $\{x_n\}$ is constructed from Standard GD converges, then the limit is not a generalised saddle point. The main idea is that then the map $x\mapsto x-\delta \nabla f(x)$ is a diffeomorphism, and hence we can use the Stable-Center manifold theorem in dynamical systems (cited as Theorem 4.4 in \cite{lee-simchowitz-jordan-recht}). For to deal with the case where the set of critical points of the function is uncountable, the new idea in \cite{panageas-piliouras} is to use Lindel\"off lemma that any open cover of an open subset of $\mathbb{R}^m$ has a countable subcover. Note that here the convergence of $\{x_n\}$ is important, otherwise one may not be able to use the Stable-Center manifolds. However, for convergence of $\{x_n\}$, one has to use results about convergence of Standard GD (such as Proposition 12.6.1 in \cite{lange}), and needs to assume more, as seen from Example 2.14 in \cite{truong-nguyen}. Therefore, Proposition 4.9 in \cite{lee-simchowitz-jordan-recht} is not valid as stated. 

There are other variants of GD which are regarded as state-of-the-art algorithms in DNN such as Momentum and NAG, Adam, Adagrad, Adadelta, and RMSProp  (see an overview in \cite{ruder}). Some of these variants (such as Adagrad and Adadelta) allow choosing learning rates $\delta _n$ to decrease to $0$ (inspired by Stochastic GD, see next paragraph) in some complicated manners which depend on the values of gradients at the previous points  $x_0,\ldots ,x_{n-1}$. However, as far as we know, convergence for such methods are not yet available beyond the usual setting such as in Proposition 12.6.1 in \cite{lange}. Indeed, as seen in the next paragraph, for the Stochastic GD, such assumptions are only enough to prove the convergence of $\{\nabla f(x_n)\}$ to $0$, and not of $\{x_n\}$ itself, even if one assumes that $f\in C^{1,1}_L$.   

Stochastic GD is the default method used to justify the use of GD in DNN, which goes back to Robbins and Monro, see \cite{bottou-etal}. The most common version of it is to assume that we have a fixed cost function $F$ (as in the deterministic case), but we replace the gradient $\nabla _{\kappa}F(\kappa _n)$ by a random vector $v_n$ (here the random variables are points in the dataset, see also Inexact GD), and then show the convergence in probability of the sequence of values $F(\kappa _n)$ and of gradients $\nabla _{\kappa}F(\kappa _n)$ to $0$ (in application the random vector $v_n$ will be $\nabla _{\kappa }F_{I_n}(\kappa _n)$). However, the assumptions for these convergence results (for $F(\kappa _n)$ and $\nabla _{\kappa}F(\kappa _n)$) to be valid still require those in the usual setting as in Proposition 12.6.1 in \cite{lange}, in particular requiring that $f\in C^{1,1}_L$ and the learning rate is small compared to $1/L$, and in addition requiring that $f$ is {\bf strongly convex}. In the case where there is noise, the  following additional conditions on the learning rates are needed \cite{robbins-monro}:
\begin{equation}
\sum _{n\geq 1}\delta _n=\infty,~ \sum _{n\geq 1}\delta _n^2<\infty,
\label{EquationStochasticGD}\end{equation} 
under which convergence of $\nabla _{\kappa}F(\kappa _n)$ to $0$ is shown. However, in Standard GD, which is a popular version in DNN, all the learning rates are the same and hence condition  $\sum _{n\geq 1}\delta _n^2<\infty$ is violated. Moreover, showing that the gradients $\nabla _{\kappa}F(\kappa _n)$ converge to $0$ is far from proving the convergence of $\kappa _n$ itself. In the original paper of Robbins and Monro, one can also see that it is required that the equation one wants to solve $M(x)=0$ has a unique solution. 

There is also a method proposed by Wolfe, which is close to Backtracking GD. The original definition in \cite{wolfe}) is very complicated, and consists of choosing at each step one of 5 choices listed in Definition on page 228 in that paper. The modern version of Wolfe's conditions consists of two assumptions, one is Armijo's condition in Backtracking GD, and the other is so-called curvature condition (which is only a half of condition v) in Wolfe's paper). It is shown that if $f$ is a $C^1$ function which is {\bf bounded from below}, then a positive $\delta _n$ can be chosen to satisfy these two Wolfe's conditions. Even though requiring more conditions than Backtracking GD, the best result so far using Wolfe's conditions requires that $f$ is in $C^{1,1}_L$  to show  the convergence of $\nabla f(z_n)$ to $0$ (G.  Zoutendijk's result, see \cite{nocedal-wright}). Hence, as evident from the next two paragraphs and results proven in this paper, one can conclude that as current, Backtracking GD, even though simpler than Wolfe's conditions, is theoretically proven better. Recently,  Wolfe's conditions are implemented in DNN \cite{mahrsereci-hennig}. 

For (discrete) Backtracking GD, it is known in literature before \cite{truong-nguyen} that any cluster point of the sequence $\{x_n\}$ is a critical point of $f$ (see \cite{bertsekas}), for all $C^1$ functions $f$. It should be noted that the terminology "limit points" in \cite{bertsekas} actually means cluster points, and hence does not mean that the sequence converges - as the word "limit" may suggest. Moreover, if $f$ is real analytic, then \cite{absil-mahony-andrews} proves convergence of the sequence $\{x_n\}$ without additional assumptions. However, the real analyticity assumption is quite restrictive, and not preserved under small perturbations.   

In \cite{truong-nguyen}, we proved convergence of Backtracking GD under the assumption that $f$ is $C^1$ and has at most countably many critical points. We obtained as a corollary - as far as we know not mentioned in previous literature - the convergence of the method using Wolfe's conditions for cost functions in $C^{1,1}_L$ and which has at most countably many critical points.  The assumption of having at most countably many critical points is satisfied by all Morse functions, i.e. functions whose all critical points are non-degenerate. Note that Morse functions are open and dense in the set of all functions, and hence are preserved under small perturbations. Hence, when using Backtracking GD, basically one does not need to check any condition. There are two main points which separate the proofs in \cite{truong-nguyen} and those in previous literature. The first is we need to find a new way to prove that either $\lim _{n\rightarrow\infty}||x_{n+1}-x_n||=0$ or $\lim _{n\rightarrow\infty}||x_n||=\infty$, given that no bootstrap techniques are available as in the case of $C^{1,1}_L$ cost functions. The second is that we introduced into the study of optimisation on $\mathbb{R}^k$ the real projective space $\mathbb{P}^k$. For a general $C^1$ function $f$, we also proved a type of avoiding of saddle points for Backtracking GD, whose conclusion is weaker than that in  \cite{lee-simchowitz-jordan-recht, panageas-piliouras} and in Theorems \ref{TheoremMain} and Theorem \ref{Theorem1}. In \cite{truong-nguyen} we also provided a heuristic argument showing that in the long term Backtracking GD will become a finite union of Standard GD processes. We also proposed new {\bf optimization} algorithms to {\bf automatically} find learning rates, working very well across many different network architectures and mini-batch sizes. These new algorithms can be integrated into existing DNN architectures to improve stability and reproducibility of those models. The source code of our algorithms can be found in \cite{mbtoptimizer}. 
 
To summarise, at the moment, Backtracking GD is theoretically proven better than other GD methods, concerning guarantee of convergence to a critical point and convergence to mimima.

\end{document}